\documentclass[11pt]{article}
\usepackage{amsmath, amsfonts, amssymb, amsthm,epsfig}
 \usepackage[ps2pdf,linktocpage,bookmarks,bookmarksnumbered,bookmarksopen]{hyperref}
\hypersetup{%
   pdftitle   ={},
   pdfsubject ={Manuscript},
   pdfauthor  ={Frank Aurzada and Christoph Baumgarten},
   pdfkeywords={}
      }

\let\BFseries\bfseries\def\bfseries{\BFseries\mathversion{bold}} 
\DeclareMathSymbol{\leqslant}{\mathalpha}{AMSa}{"36} 
\DeclareMathSymbol{\geslant}{\mathalpha}{AMSa}{"3E} 
\DeclareMathSymbol{\eset}{\mathalpha}{AMSb}{"3F}     
\newcommand{\abs}[1]{\lvert#1\rvert}

\newcommand{\ind}{1\hspace{-0.098cm}\mathrm{l}}

\newcommand{\be}{\begin{eqnarray*}}
\newcommand{\ee}{\end{eqnarray*}}
\newcommand{\ben}{\begin{eqnarray}}
\newcommand{\een}{\end{eqnarray}}

\theoremstyle{plain}
\newtheorem{thm}{Theorem}
\newtheorem{lem}[thm]{Lemma}
\newtheorem{prop}[thm]{Proposition}

\theoremstyle{definition}

\newtheorem{remark}[thm]{Remark}

\renewenvironment{proof}[1][] {\smallskip \noindent {\bf Proof#1.} }{\hspace*{\fill}$\square$\medskip\par}
\def\P{{\bf {\mathbb{P}}}}
\newcommand{\pr}[1]{\P\left[#1\right]}

\newcommand{\E}[1]{\mathbb{E}\left[#1\right]}

\def\d{{\rm d}}

\newcommand{\set}[1]{\left\lbrace #1 \right\rbrace}

\begin{document}
\title{{\sc Persistence of fractional Brownian motion with moving boundaries and applications }}
\author{\renewcommand{\thefootnote}{\arabic{footnote}} {\sc Frank Aurzada}\footnotemark[1]\\
\renewcommand{\thefootnote}{\arabic{footnote}}{\sc Christoph Baumgarten}\footnotemark[1]}
\date{\today}

\footnotetext[1]{
Technische Universit\"at Braunschweig, Institut f\"ur Mathematische Stochastik, 
Pockelsstr.\ 14, 38106 Braunschweig, Germany,
{\sl f.aurzada@tu-braunschweig.de}, 
{\sl baumgart@math.tu-berlin.de}
}
\maketitle
%

\begin{abstract}
\noindent
We consider various problems related to the persistence probability of fractional Brownian motion (FBM), which is the probability that the FBM $X$ stays below a certain level until time $T$.\\
Recently, Oshanin et al.\ (\cite{o-r-s:2012}) study a physical model, where persistence properties of FBM are shown to be related to scaling properties of a quantity $J_N$, called steady-state current. It turns out that for this analysis it is important to determine persistence probabilities of FBM with a moving boundary.\\
We show that one can add a boundary of logarithmic order to a FBM without changing the polynomial rate of decay of the corresponding persistence probability which proves a result needed in \cite{o-r-s:2012}. Moreover, we complement their findings by considering the continuous-time version of $J_T$. Finally, we use the results for moving boundaries in order to improve estimates by Molchan (\cite{molchan:1999a}) concerning the persistence properties of other quantities of interest, such as the time when a FBM reaches its maximum on the time interval $(0,1)$ or the last zero in the interval $(0,1)$.  

\end{abstract}

\section{Introduction}
Given a real-valued stochastic process $(Z_t)_{t \ge 0}$, consider the persistence or survival probability up to time $T$ given by
\begin{equation*}
   p(T) := \pr{ Z_t \leq 1, \forall t \in [0,T]}, \quad T > 0.
\end{equation*}
For many relevant stochastic processes, it decreases polynomially (modulo terms of lower order), i.e.\ $p(T) = T^{-\theta + o(1)}$ as $T \to \infty$, and $\theta > 0$ is called the persistence or survival exponent. Persistence probabilties are related to many problems in physics and mathematics, see the surveys \cite{majumdar:1999} and \cite{aurzada-simon:2012} for a collection of results, applications, and examples.\\
In this article, we discuss persistence probabilities related to fractional Brownian motion (FBM). Recall that FBM with Hurst index $H \in (0,1)$ is a centered Gaussian process $(X_t)_{t \in \mathbb{R}}$ with covariance 
\[
  \E{X_s X_t} = \frac{1}{2} \left( \abs{s}^{2H} + \abs{t}^{2H} - \abs{t-s}^{2H} \right), \quad s,t \in \mathbb{R}.
\]
We remark that $X$ has stationary increments and is self-similar of index $H$, i.e.\ $(X_{ct})_{t \in \mathbb{R}}$ and $(c^H X_t)_{t \in \mathbb{R}}$ have the same distribution for any $c > 0$. Let us remark that $X$ is non-Markovian unless $H=1/2$ (see e.g.\ \cite{mandelbrot-vanness:1968}). \\
Since the behavior of many dynamical systems exhibits long-range correlations, one observes so-called \textit{anomalous dynamics} which are typically characterized by a nonlinear growth in time (i.e.\ $\E{X_t^2} \propto t^{2H}$ where $H \neq 1/2$) where $X$ models the evolution of the corresponding quantity (\cite{bouchoud-georges:1990}). In order to take such features into account, FBM has been proposed in \cite{mandelbrot-vanness:1968} in 1968. For instance, FBM has been used in polymers models (\cite{zoia-rosso-majumdar:2009, w-f-c-v:2012}) and in finance to describe longe-range dependence of stock prices and volatility (\cite{comte-renault:1998, oksendal:2007}). We also refer to \cite{eliazar-klafter:2008} and \cite{eliazar-shlesinger:2012} where the emergence of FBM in certain complex systems is investigated.\\
The study of persistence for this process has been motivated by the analysis of Burgers equation with random inital conditions (\cite{sinai:1992-a}) and the linear Langevin equation (\cite{krug-et-al:1997}). Sina{\u\i} also derived estimates on the persistence probability in a subsequent article (\cite{sinai:1997}). The exponent was shown to equal $\theta = 1 - H$ by \cite{molchan:1999a}, where $H$ is the Hurst parameter of the FBM. The estimates on the persistence probability have recently been improved in \cite{aurzada:2011}, who showed the following: there is a constant $c = c(H) > 0$ such that, for $T$ large enough,
\begin{equation}\label{eq:survival_prob_FBM}
   T^{-(1-H)} (\log T)^{-c} \precsim \pr{X_t \le 1, 0 \le t \le T} \precsim T^{-(1-H)} (\log T)^{c}, \quad T \to \infty.
\end{equation}
The notation $f(x) \precsim g(x)$ as $x \to x_0$ means that $\limsup_{x \to x_0} f(x)/g(x) < \infty$, whereas we write $f(x) \sim g(x)$ ($x \to x_0$) if $f(x)/g(x) \to 1$ as $x \to x_0$. However, it is still an open problem to show that $p(T) \asymp T^{-(1-H)}$ where $f(T) \asymp g(T)$ means that the ratio $f(T)/g(T)$ is bounded away from zero and infinity for large values of $T$. Note that in view of the self-similarity, \eqref{eq:survival_prob_FBM} translates into
\begin{equation}\label{eq:small_value_max_FBM}
   \abs{\log \epsilon}^{-c} \epsilon^{(1 - H)/H} \precsim \pr{X_t \le \epsilon, t \in [0,1]} \precsim \abs{\log \epsilon}^{c} \epsilon^{(1 - H)/H}, \quad \epsilon \downarrow 0.
\end{equation}
Let us remark that the persistence exponent of another non-Markovian process with similar properties, namely self-similarity and stationarity of increments, has been computed recently in \cite{castell-et-al:2012}, confirming results in \cite{redner:1997,majumdar:2003}.\\
The main motivation of this article comes from a physical model involving FBM that has been studied recently in \cite{o-r-s:2012} as an extension to the Sina{\u\i} model. If $(X_t)_{t \ge 0}$ denotes a FBM with Hurst index $H$, the authors are interested in the asymptotics of the $k$-th moment $\E{J_N^k}$ of the quantity $J_N$, called the steady-state current $J_N$ through a finite segment of length $N$, given by
\[
   J_N := \frac{1}{2} \, \left( 1+ \sum_{n=1}^{N-1} \exp(X_n) \right)^{-1}.
\]
Oshanin et al.\ find that $\E{J_N^k} = N^{-(1-H) + o(1)}$ as $N \to \infty$ for any $k > 0$. In particular, the exponent is independent of $k$. In order to prove the lower bound, the authors need the following estimate: If $Y_0,Y_1 > 0$ is some constant, then
\begin{equation}\label{eq:FBM_moving_bound_in_physics_paper}
    N^{-(1-H)} (\log N)^{-c} \precsim \pr{X_n \le Y_0 - Y_1 \log(1+n), \, \forall n = 1,\dots,N}, \qquad N \to \infty.
\end{equation}
In general, the following question arises: What kind of functions $f$ are admissible such that $\pr{X_t \le f(t), \forall t \in [0,T]} = T^{-(1-H) + o(1)}$, i.e.\ what kind of moving boundaries $f$ do not change the persistence exponent of a FBM? Given the increasing relevance of FBM for various applications, it is important to understand such questions since they convey information about the path behavior of FBM. In this article, we take a further step in this direction. Let us now briefly summarize our main results.
\begin{itemize}
   \item We study the persistence probability of FBM involving a moving boundary that is allowed to increase or decrease like some power of a logarithm. Our results show that the presence of such a boundary does not change the persistence exponent of FBM, and \eqref{eq:FBM_moving_bound_in_physics_paper} will follow as a special case. 
\begin{prop}\label{prop:FBM_log_boundary}
   Let $Y_0,Y_1>0$ and $X$ denote a FBM with Hurst index $H \in (0,1)$.
\begin{enumerate}
   \item  For any $\gamma \ge 1$, there is a constant $c = c(H,\gamma) >0$ such that, for $T$ large enough,
$$
T^{-(1-H)} (\log T)^{-c} \precsim \pr{X_s\leq Y_0 - Y_1 (\log (1 + s))^\gamma, 0\leq  s\leq T} \precsim T^{-(1-H)}.
$$
\item  For any $\gamma > 0$, there is a constant $c = c(H,\gamma) > 0$ such that, for $T$ large enough,
$$
T^{-(1-H)} (\log T)^{-c} \precsim \pr{X_s\leq Y_0 + Y_1 (\log (1 + s))^\gamma, 0\leq  s\leq T} \precsim T^{-(1-H)} (\log T)^{c}.
$$
\end{enumerate}
\end{prop}
\item Considering the continuous-time version of $J$, we prove the following result:
\begin{prop}\label{prop:moment_disorder_current_functional}
  Set
\[
   J_T := \left( \int_0^T e^{X_s} \d s \right)^{-1}, \quad T > 0.
\]
For any $k > 0$, there is $c = c(H) > 0$ such that 
\begin{equation}
   T^{-(1-H)} (\log T)^{-c} \precsim \E{J_T^k} \precsim T^{-(1-H)} (\log T)^c, \qquad T \to \infty.
\end{equation}
\end{prop}
Solving the case $k=1$ was actually the key to the computation of the persistence exponent in \cite{molchan:1999a} where it is shown that $\E{J_T} \sim C T^{-(1-H)}$ for some constant $C > 0$. Our proof is based on estimates of the persistence probability of FBM in \cite{aurzada:2011}, an estimate on the modulus of continuity of FBM in \cite{scheutzow:2009} and Proposition~\ref{prop:FBM_log_boundary}.
\item Finally, we discuss various related quantities such as the time when a FBM reaches its maximum on the time interval $(0,1)$, the last zero in the interval $(0,1)$ and the Lebesgue measure of the set of points in time when $X$ is positive on $(0,1)$. If $\xi$ denotes any of these quantities, we are interested in the probability of small values, i.e.\ $\pr{\xi < \epsilon}$ as $\epsilon$ goes to zero. In Proposition~\ref{prop:FBM_related_quantities} below, we improve the estimates given in \cite{molchan:1999a}.
\end{itemize}
These issues are addressed in Sections~\ref{sec:fbm_mov}, \ref{sec:proof_fbm_moment_of_funct} and \ref{sec:fbm_rel_quant} respectively. 
\section{Survival probability of FBM with moving boundaries}\label{sec:fbm_mov}
In this section, we prove Proposition~\ref{prop:FBM_log_boundary}. We need to distinguish between increasing and decreasing boundaries. Let us begin with a simple general upper bound on the probabilitiy that a FBM stays below a function $f$ until time $T$ when $f(x) \to -\infty$ as $x \to \infty$.
\begin{lem}\label{lem:FBM_barrier_upper_bound}
Let $f$ be some measurable function such that there is a constant $b > 0$ such that $\int_0^\infty e^{b f(s)} \, \d s < \infty$. Then
$$
 \pr{X_s\leq f(b^{1/H} s), 0 \le  s \le T} \precsim T^{-(1-H)}.
$$
\end{lem}
\begin{proof}
Recall from \cite[Statement~1]{molchan:1999a} that
$$
\lim_{T\to\infty} T^{1-H} \E{ J_T } \in(0,\infty).
$$
Therefore, there is a constant $c > 0$ such that, for $T$ large enough,
\begin{align*}
c \, T^{-(1-H)} & \geq \E{ \frac{1}{\int_0^T e^{X_s} \d s} } \geq \E { \frac{1}{\int_0^T e^{X_s} \d s} \ind_{\{ X_s \leq b f(s), s\leq T\} } } \\
             & \geq \frac{1}{\int_0^T e^{b f(s)} \d s}\, \pr{ X_s \le b f(s), 0 \le s \le T} \\
             & \geq \frac{1}{\int_0^\infty e^{b f(s)} \d s}\,  \pr{ X_{b^{-1/H} s} \leq f(s), 0 \le s \leq T} \\
	     & = C(b) \, \pr{ X_s \leq f(b^{1/H} s), 0 \le s \leq b^{-1/H} T},
\end{align*}
and the lemma follows.
\end{proof}
The next lemma provides a lower bound on the survial probability if the function $f$ does not decay faster than some power of the logarithm.
\begin{lem}\label{lem:FBM_barrier_lower_bound}
   Let $f$ be some measurable, locally bounded function such that $f$ is positive in a vicinity of $0$. Assume that there are constants $T_0, K, \alpha > 0$ such that $f(T) \ge -K (\log T)^\alpha$ for all $T \ge T_0$. Then there is a constant $c >0$ such that
\[
   \pr{X_s \le f(s), 0 \le s \le T} \succsim T^{-(1-H)} (\log T)^{-c}.
\]
\end{lem}
\begin{proof}
   Set $g(T):= \pr{X_s\leq f(s), 0 \leq  s\leq T}$ and fix $s_0>0$ (to be chosen later). Since $\E{X_s X_t}\geq 0$ for all $t,s\geq 0$, Slepian's lemma (see \cite{slepian:1962}) yields
\begin{align*}
g(T)\geq \pr{X_s\leq f(s), 0\leq s\leq s_0 (\log T)^{\alpha/H}} \cdot \pr{X_s\leq f(s), s_0 (\log T)^{\alpha/H} \leq s \leq T}.
\end{align*}
Note that
\begin{align}
 & \pr{X_s \leq f(s), s_0 ( \log T)^{\alpha/H} \leq s \leq T}  \notag \\ \notag
&\quad = \pr{ X_{(\log T)^{\alpha / H} s}\leq f\left((\log T)^{\alpha / H} s\right), s_0 \leq s \leq T/(\log T)^{\alpha/H}} \\  \notag
&\quad =\pr{ (\log T)^\alpha X_{s} \leq f\left((\log T)^{\alpha/H} s\right), s_0 \leq s \leq T/(\log T)^{\alpha/H}}\\ 
&\quad =\pr{ X_{s}\leq \frac{f((\log T)^{\alpha/H} s) }{(\log T)^\alpha}, s_0 \leq s \leq T/(\log T)^{\alpha/H}}. \label{eqn:aftre}
\end{align}
Certainly, for all $T$ large enough, 
\begin{align*}
&\inf_{s \in [s_0, T / (\log T)^{\alpha/H}]} \frac{f((\log T)^{\alpha/H} s) }{(\log T)^\alpha} \\
&\quad = \inf_{s \in [s_0, T / (\log T)^{\alpha/H}]} \frac{f((\log T)^{\alpha/H} s) }{(\log [( \log T)^{\alpha/H} s ] )^\alpha} \cdot \frac{( \log [( \log T)^{\alpha/H} s ] )^\alpha}{(\log T)^\alpha} \ge -K.
\end{align*}
Thus, the term in (\ref{eqn:aftre}) can be estimated from below by
\begin{equation}\label{eq:FBM_arbitrary_constant_barrier}
   \pr{ X_{s}\leq -K, s_0 \leq s \leq T/(\log T)^{1/H}}.
\end{equation}
Let us first consider the case $H \ge 1/2$. Recall that the increments of FBM are positively correlated if and only if $H \ge 1/2$, so using Slepian's lemma in the second inequality, we obtain the following lower bound for the term in \eqref{eq:FBM_arbitrary_constant_barrier}:
\begin{align*}
   \pr{X_s \le -K, s_0 \le s \le T} &\ge \pr{X_{s_0} \le -(K+1), \sup_{s \in [s_0,T]} X_s - X_{s_0} \le 1} \\
&\ge \pr{X_{s_0} \le -(K+1)} \cdot \pr{ \sup_{s \in [s_0,T]} X_s - X_{s_0} \le 1} \\
&\ge c(s_0,K) \pr{X_s \le 1, s \in [0,T]}.
\end{align*}
Hence,
\[
 g(T) \ge c(s_0,K) g(s_0 (\log T)^{\alpha/H}) \cdot \pr{X_s \le 1, s \in [0,T]},
\]
and \eqref{eq:survival_prob_FBM} implies that there is $c > 0$ such that, for all large $T$,
\[
   g(T) \ge g(s_0 (\log T)^{\alpha/H}) T^{-(1-H)} (\log T)^{-c}.
\]
Let us now prove that a similar inequality also holds if $H < 1/2$. In this case, we cannot use Slepian's inequality since the increments of FBM are negatively correlated. Applying \cite[Lemmma~5]{aurzada:2011} (and the specific choice of $s_0$ there), the term in \eqref{eq:FBM_arbitrary_constant_barrier} is lower bounded by
$$
\pr{ X_{s} \leq 1, 0 \leq s \leq k\, T/(\log T)^{1/H} (\log \log T)^{1/(4H)}} (\log T)^{-o(1)},
$$
where $k$ is some constant. Finally, by \eqref{eq:survival_prob_FBM}, this term admits the lower bound $T^{-(1-H)} (\log T)^{-c}$
with some appropriate constant $c>0$ and all $T$ large enough. Thus, we have seen that
\begin{equation}
g(T)\geq g(s_0 ( \log T)^{1/H} ) T^{-(1-H)} (\log T)^{-c}
\end{equation}
for some constants $s_0,c>0$.\\
If we combine this result with the case $H \ge 1/2$, this shows that for any $H \in (0,1)$, there are constants $c = c(H), \beta = \beta(H), s_0 = s_0(H) > 0$ such that
\begin{equation}\label{eq:iteration_g}
g(T)\geq g(s_0 ( \log T)^\beta ) T^{-(1-H)} (\log T)^{-c}.
\end{equation}
Using this inequality iteratively, we will prove the preliminary estimate $g(T) \ge T^{-\theta_1}$ for some $\theta_1 > 1 - H$ and all $T$ large enough. Once we have this estimate, \eqref{eq:iteration_g} shows that 
\[
   g(T) \succsim (\log T)^{-(\theta_1 \beta + c)} \, T^{-(1-H)}, \qquad T \to \infty, 
\]
and the proof is complete for all $H \in (0,1)$. \\
Let us now establish the preliminary lower bound. \eqref{eq:iteration_g} implies that if $\beta_1 > \beta$ and $\theta > 1-H$, there is a constant $T_0 \ge 1$ such that
\begin{equation}\label{eq:iteration_g_simple}
   g(T) \ge g((\log T)^{\beta_1}) \cdot T^{-\theta}, \qquad T \ge T_0.
\end{equation}
The idea is to iterate this inequality until $\log(\log(\dots)^{\beta_1})^{\beta_1}$ is smaller than some constant. As we will see, the number of iterations that are needed is very small and merely leads to a term of logarithmic order. Since each iteration is valid only for large values of $T$ depending on the number of iterations, and the number of iterations is itself a function of $T$, some care is needed to perform this step. To this end, fix $\beta_2 > \beta_1$ and set $T_0' := \max\set{\log(T_0)/\beta_2, \beta_2^{\beta_1/(\beta_2 - \beta_1)}}$.
Define $\log^{(1)} x = \log x$ for $x > x_1 = 1$ and $\log^{(i)}x = \log^{(i-1)}(\log x)$ for $x > x_i :=\exp(x_{i-1})$. For any $j \ge 1$ and $T > 0$, the following implication holds:
\begin{equation}\label{eq:iteration_g_with_log}
   \log^{(j+1)} T \ge T_0' \Longrightarrow g((\log^{(j)} T)^{\beta_2}) \ge g((\log^{(j+1)} T)^{\beta_2}) \cdot (\log^{(j)} T)^{-\theta \beta_2}.
\end{equation}
Indeed, note that $\log^{(j+1)} T \ge T_0'$ translates into 
\[
  (\log^{(j)} T)^{\beta_2} \ge T_0 \quad \text{ and } \quad \beta_2^{\beta_1} (\log^{(j+1)} T)^{\beta_1} \le  (\log^{(j+1)} T )^{\beta_2}.
\]
Hence, in view of \eqref{eq:iteration_g_simple}, we find that
\begin{align*}
   g((\log^{(j)} T)^{\beta_2}) &\ge g((\log( (\log^{(j)} T)^{\beta_2}))^{\beta_1}) \cdot (\log^{(j)} T)^{-\beta_2 \theta} \\
 &= g(\beta_2^{\beta_1} (\log^{(j+1)} T)^{\beta_1}) \cdot (\log^{(j)} T)^{-\beta_2 \theta} \\
&\ge g( (\log^{(j+1)} T )^{\beta_2}) \cdot (\log^{(j)} T)^{-\beta_2 \theta},
\end{align*}
so \eqref{eq:iteration_g_with_log} follows. Denote by $a(T) := \min\set{ n \in \mathbb{N} : \log^{(n)} T \le T_0'}$. By definition, $\log^{(a(T))} T \le T_0' < \log^{(a(T) - 1)} T$, so we can apply \eqref{eq:iteration_g_with_log} iteratively for all $j \le a(T) - 2$ to obtain that
\begin{align}
   g((\log T)^{\beta_2}) &\ge g((\log^{(2)} T)^{\beta_2}) (\log T)^{-\beta_2 \theta} \ge \dots \notag \\
&\ge g((\log^{(a(T) -1)} T)^{\beta_2}) \prod_{j=1}^{a(T) - 2} (\log^{(j)} T)^{-\beta_2 \theta} \notag\\
&\ge g(e^{T_0' \beta_2}) \prod_{j=1}^{a(T) - 2} (\log^{(j)} T)^{-\beta_2 \theta}, \label{eq:iteration_g_with_log_2}
\end{align}
which holds for all $T \ge \exp(\exp(\exp(T_0')))$, i.e.\ such that $a(T) \ge 3$.
Finally,
\begin{align*}
   \prod_{j=1}^{a(T)-2} (\log^{(j)} T)^{-\beta_2 \theta} \ge (\log T)^{-\beta_2 \theta} \cdot (\log^{(2)} T)^{-\beta_2 \theta a(T)}.
\end{align*}
In view of 
\[
  a(T) = 1 + a(\log T) = j + a(\log^{(j)} T),
\]
which holds for any $j \in \mathbb{N}$ and $T$ large enough and the simple observation that $a(T) \le T$, we obtain that $a(T) = o(\log^{(j)} T)$ for any $j \in \mathbb{N}$. Hence, for all $T$ large enough,
\begin{align}\label{eq:iteration_g_auxil}
   (\log^{(2)} T)^{-\beta_2 \theta \, a(T)} &\ge (\log^{(2)} T)^{-\beta_2 \theta \log^{(3)} T} = \exp \left( -\beta_2 \theta (\log^{(3)} T)^2 \right) \notag \\
&\ge \exp \left( - \log^{(2)} T \right) = (\log T)^{-1}.
\end{align} 
Combining \eqref{eq:iteration_g_simple}, \eqref{eq:iteration_g_with_log_2} and \eqref{eq:iteration_g_auxil}, we conclude that $T^{-\theta_1} \precsim g(T)$ for any $\theta_1 > \theta$.
\end{proof}
Combining Lemma~\ref{lem:FBM_barrier_upper_bound} and Lemma~\ref{lem:FBM_barrier_lower_bound}, we obtain part 1 of Proposition~\ref{prop:FBM_log_boundary}. \\
\begin{proof}[ of part 1 of Proposition~\ref{prop:FBM_log_boundary}]\\
\textit{Lower bound:} With $f(s) :=  Y_0 - Y_1 (\log (1 + s))^\gamma \ge -2 Y_1 (\log (1 + s))^\gamma$ for all large $s$, the lower bound follows directly from Lemma~\ref{lem:FBM_barrier_lower_bound}.\\
\textit{Upper bound:} If $\gamma > 1$, we can directly apply Lemma~\ref{lem:FBM_barrier_upper_bound} with $f(s) :=  Y_0 - Y_1 (\log (1 + s))^\gamma$ and $b = 1$ to obtain the upper bound. \\
If $\gamma = 1$, take $b > 0$ such that $b Y_1 > 1$ and set $f(s) :=  Y_0 - Y_1 (\log (1 + b^{-1/H}s))$, so that $\int_0^\infty e^{b f(s)} \, \d s < \infty$ and by Lemma~\ref{lem:FBM_barrier_upper_bound},
\[
    T^{-(1-H)} \succsim \pr{X_s \le f(b^{1/H}s), 0 \le s \le T} = \pr{Y_0 - Y_1 \log (1 + s), 0 \leq  s \leq T }.
\]
\end{proof}
\begin{remark}
\begin{enumerate}
   \item We remark that the removal of the boundary by a change of measure argument (Cameron-Martin-formula) results in less precise estimates of the form 
\[
   T^{-(1-H)} e^{-c\sqrt{\log T}} \precsim \pr{X_s\leq Y_0 - Y_1 (\log (1 + s))^\gamma, 0\leq  s\leq T} \precsim  T^{-(1-H)} e^{c\sqrt{\log T}},
\]
see \cite{aurzada-dereich:2011}, \cite{molchan:1999a} or \cite{molchan:2012}.
\item In view of the results for Brownian motion (i.e.\ $H=1/2$, see \cite{uchiyama:1980}), it is reasonable to expect that the upper bound in part 1 of Proposition~\ref{prop:FBM_log_boundary} has the correct order. 
\item The restriction $\gamma \ge 1$ is necessary in order to apply Lemma~\ref{lem:FBM_barrier_upper_bound}. However, for any $\gamma > 0$, \eqref{eq:survival_prob_FBM} immediately implies the following weaker bound:
\[
   \pr{X_s \le Y_0 - Y_1 (\log(1+s))^\gamma, 0 \le s \le T} \le \pr{X_s \le Y_0, 0 \le s \le T} \precsim T^{-(1-H)} (\log T)^c. 
\]
\item Let $f(x) = Y_0 - Y_1 \log(1+x)$. Trivially, if we consider discrete time, 
\[
   \pr{X_k \le f(k), k=1,\dots,N} \ge \pr{X_s \le f(s), 0 \le s \le N} \succsim N^{-(1-H)} \log(N)^{-c}.
\]
This estimate is needed in \cite{o-r-s:2012} (see Eq.\ (15) there) when proving a lower bound for $\E{J_N^k}$. 
\end{enumerate}
\end{remark}
Clearly, Lemma~\ref{lem:FBM_barrier_lower_bound} is only applicable if the boundary $f$ satisfies $f(x) \to -\infty$ as $x \to \infty$. It is natural to suspect that the persistence exponent does not change if we introduce a barrier that increase like some power of a logarithm.  This is part 2 of Proposition~\ref{prop:FBM_log_boundary} which follows from the next lemma:
\begin{lem}
   Let $f \colon [0,\infty) \to \mathbb{R}$ denote a measurable function such that there are $A,\delta > 0$ such that $f(x) \ge A$ for all $x \in [0,\delta]$ and $f(x) \ge -1/A$ for all $x \ge 0$. Moreover, we assume that there are $\alpha, T_0 > 0$ such that $f(x) \le (\log x)^{\alpha}$ for all $x \ge T_0$.
Then there is a constant $c > 0$ such that
\[
   T^{-(1-H)} (\log T)^{-c} \precsim \pr{X_s \leq f(s), 0 \le  s \le T} \precsim T^{-(1-H)} (\log T)^c.
\]
\end{lem}
\begin{proof}
   \textit{Lower bound:} 
Note that we can directly apply Lemma~\ref{lem:FBM_barrier_lower_bound} directly since $f(x) \ge A$ on $[0,\delta]$ and $f(x) \ge -1/A$ on $[0,\infty)$.\\ 
\textit{Upper bound:} Note that
\begin{align*}
   \pr{X_s \le f(s), 0 \le s \le T} &\le \pr{X_s \le f(s), T_0 \le s \le T} \\
&\le \pr{X_s \le (\log s)^\alpha, T_0 \le s \le T} \\
&\le \pr{X_s \le (\log (2+T))^\alpha, T_0 \le s \le T} \\
&\le \frac{ \pr{X_s \le (\log (2+T))^\alpha, 0 \le s \le T}}{ \pr{X_s \le ( \log (2+T) )^\alpha, 0 \le s \le T_0}} \\
&\sim \pr{X_s \le ( \log (2+T) )^\alpha, 0 \le s \le T}, \quad T \to \infty.
\end{align*}
We have used Slepian's inequality in the last inequality. Using once more the self-similarity and \eqref{eq:survival_prob_FBM}, the upper bound follows.
\end{proof}
\section{Proof of Proposition~\ref{prop:moment_disorder_current_functional}} \label{sec:proof_fbm_moment_of_funct}
We are now ready to prove Proposition~\ref{prop:moment_disorder_current_functional}. The lower bound follows easily from our result on moving boundaries in Proposition~\ref{prop:FBM_log_boundary}, whereas the proof of the upper bound is more involved. 
\begin{proof}[ of Proposition~\ref{prop:moment_disorder_current_functional}] \\
\textit{Lower bound:} Let $\gamma > 1$.
 \begin{align*}
      \E{J_T^k} &\ge \E{ \left( \int_0^T e^{X_s} \d s \right)^{-k} ; \lbrace X_s \le 1 - (\log(1+s))^\gamma, \, \forall s \in [0,T] \rbrace } \\
&\ge  \left( \int_0^T \frac{e}{(1+s)^\gamma} \d s \right)^{-k}  \pr{X_s \le 1 - (\log(1+s))^\gamma, \, \forall s \in [0,T]}.
 \end{align*}
The lower bound now follows by part 1 of Proposition~\ref{prop:FBM_log_boundary}.\\
\textit{Upper bound:} Let $H/2 < \gamma < H$ and fix $a$ such that $a > 2/H > 1/\gamma$ and $\gamma < H -1/a$. Self-similarity and stationarity of increments imply for all $s,t \in [0,1]$ that
\[
   \E{\abs{X_t - X_s}^a} = \E{\abs{X_{\abs{t-s}}}^a} = \abs{t-s}^{aH} \E{\abs{X_1}^a} = \abs{t-s}^{(aH-1) + 1} \E{\abs{X_1}^a}.
\]
Since $aH-1 > 0$, it follows from \cite[Lemma~2.1]{scheutzow:2009} that there is a positive random variable $S$ such that
\begin{equation}\label{eq:expect_Hoelder_constant}
   \E{S^a} \le \left( \frac{2}{1-2^{-\gamma}} \right)^a \cdot \frac{\E{\abs{X_1}^a}}{2^{aH-1-a\gamma} - 1},
\end{equation}
and for all $\epsilon \in (0,1)$,
\begin{equation}\label{eq:Hoelder_cont}
    \abs{X_t - X_s} \le S \epsilon^\gamma, \qquad \forall \, s,t \in [0,1], \abs{t-s} \le \epsilon.
\end{equation}
Write $X^*_1 := \sup \set{X_t : t \in [0,1]}$, and let $u^*$ denote a point where the supremum is attained. Using the self-similarity of $X$ and \eqref{eq:Hoelder_cont} in the second inequality, we obtain the following estimates:
\begin{align*}
   \E{J_T^k} &= \E{ \left( \int_0^1 e^{T^H X_s} T \d s \right)^{-k} } \\
&= T^{-k} \, \E{ e^{-kT^H X_1^*} \, \left( \int_0^1 e^{-T^H (X_1^* - X_s)} \d s \right)^{-k}  } \\
&\le T^{-k} \, \E{ e^{-kT^H X_1^*} \, \left( \int_{\min\set{u^* - \epsilon,0}}^{\max\set{u^* + \epsilon,1}} e^{-T^H (X_{u^*} - X_s)} \d s \right)^{-k}  } \\
&\le T^{-k} \, \E{ e^{-kT^H X_1^*} \, \left( \int_{\min\set{u^* - \epsilon,0}}^{\max\set{u^* + \epsilon,1}} e^{-T^H S \epsilon^\gamma} \d s \right)^{-k}  } \\
&\le T^{-k} \, \E{ e^{-kT^H X_1^*} \epsilon^{-k} e^{k T^H S \epsilon^\gamma} }.
\end{align*}
Set $\epsilon := \min\set{(T^H S)^{1/\gamma},1}$. Then $T^H S \epsilon^\gamma \le 1$ and $\epsilon^{-1} \le (T^H S)^{1/\gamma} + 1$, and we find that
\begin{align}
    \E{J_T^k} &\le T^{-k} e^k \E{ e^{-kT^H X_1^*} ( (T^H S)^{1/\gamma} + 1 )^k } \notag \\
&\le T^{-k} (2e)^k \left( \E{ e^{-kT^H X_1^*} S^{k/\gamma} } T^{kH/\gamma}  + \E{e^{-kT^H X_1^*}} \right). \label{eq:int_func_prelim}
\end{align}
Applying H\"{o}lder's inequality ($1/p + 1/q = 1$), we have that 
\begin{equation}\label{eq:Hoelder_estimate}
   \E{ e^{-kT^H X_1^*}  S^{k/\gamma} } \le \E{ e^{-q kT^H X_1^*}}^{1/q}  \E{S^{a} }^{1/p} \le \E{ e^{- kT^H X_1^*}}^{1/q}  \E{S^{a} }^{1/p}, 
\end{equation}
where $a = kp/\gamma$, and the last inequality holds for all $T > 0, a > 2/H$ and $H/2 < \gamma < H - 1/a$. Fix $\delta \in (0,1)$ and set
\[
   a := (\log \log T)^{-\delta} \, \log T ,\quad \gamma = H - 2/a.
\]
(Since $a = kp/\gamma$, this amounts to $p = ( H (\log \log T)^{-\delta} \, \log T + 2) / k$, $a = (kp - 2)/H$ and $\gamma = H - 2/a$).
Assume for a moment that there are constants $M,\nu \in (0,\infty)$ such that for all $a$ large enough, it holds that
\begin{equation}\label{eq:moment_estimate_SSSI}
   (\E{\abs{X_1}^a})^{1/a} \le M a^\nu.
\end{equation}
Then in view of the relations $1/p = k/(a \gamma)$ and $aH - a\gamma = 2$, we obtain that
\[
   \E{S^{a} }^{1/p} \le \frac{\left( (\E{\abs{X_1}^a})^{1/a} \right)^{k/\gamma} }{(2^{aH - 1 - a\gamma} - 1)^{1/p}} = \frac{M^{k/\gamma} a^{\nu k/\gamma} }{(2^{2 - 1} - 1)^{1/p}} = M^{k/\gamma} a^{k\nu/H + o(1)}, \quad a \to \infty.
\]
In particular, $(\E{S^{a}})^{1/p} = o((\log T)^{\eta})$ as $a \to \infty$, or equivalently, $T \to \infty$, for every $\eta > k \nu /H$. For such $\eta$, combining \eqref{eq:int_func_prelim} and \eqref{eq:Hoelder_estimate}, we find for $T$ large enough that
\begin{equation}\label{eq:upper_bound_SSSI}
   \E{J_T^k} \le 2 (2e)^k T^{kH/\gamma - k} (\log T)^\eta \E{e^{-kT^H X_1^*}}^{1/q}.
\end{equation}
By Karamata's Tauberian theorem (see \cite[Theorem~1.7.1]{b-g-t}), \eqref{eq:small_value_max_FBM} implies that (with the same $c > 0$ as in \eqref{eq:small_value_max_FBM})
\begin{equation}\label{eq:Laplace_FBM}
   \lambda^{-(1-H)/H} (\log \lambda )^{-c} \precsim \E{ e^{- \lambda X_1^*}} \precsim \lambda^{-(1-H)/H} (\log \lambda )^{c}, \qquad \lambda \to \infty.
\end{equation}
(In fact, the lower bound is easy since $\E{ e^{- \lambda X_1^*}} \ge e^{-1} \pr{X_1^* \le 1/\lambda}$. For our purposes, it is enough to note that $\E{ e^{- \lambda X_1^*}} \le \pr{X_1^* \le \log(\lambda)/\lambda} + e^{-\log \lambda}$, so the upper bound follows from \eqref{eq:small_value_max_FBM} with some $\tilde{c} > c$.) By \eqref{eq:Laplace_FBM}, we conclude that
\begin{align*}
  T^{Hk/\gamma - k} \E{ e^{- kT^H X_1^*}}^{1/q} &= T^{Hk/\gamma - k} \E{ e^{-kT^H X_1^*}}^{1 - k/(a \gamma)} \\
&\le C' T^{Hk/\gamma - k} T^{-(1-H)(1 - k/(a \gamma)) } (\log T)^{c(1 - k/(a \gamma))} \\
&= C' T^{k/(a\gamma) (aH - a\gamma + (H-1))} T^{-(1-H)} (\log T)^{c + o(1)}.
\end{align*}
Using again that $aH - a\gamma = 2$, note that by definition of $a$,
\[
   T^{k/(a\gamma) (Ha - a\gamma + (H-1))} = T^{(H+1)/(a\gamma)} = \exp( \gamma^{-1}(H+1) (\log\log T)^\delta ) = (\log T)^{o(1)}.
\]
Hence, we have shown that
\[
   \E{J_T^k} \precsim (\log T)^{\eta + o(1)} T^{-(1-H)}, \quad T \to \infty,
\]
as soon as we prove that \eqref{eq:moment_estimate_SSSI} holds. Since $X_1$ is standard Gaussian, it is well-known that $\E{\abs{X_1}^a} = 2^{a/2} \Gamma((a+1)/2)/\sqrt{\pi}$ for every $a > 0$, and therefore, it is not hard to show that $\E{\abs{X_1}^a}^{1/a} \le M \sqrt{a}$ for some $M$ and all $a$ large enough. This completes the proof.
\end{proof}
\begin{remark}
   Note that if $X$ is a self-similar process with stationary increments (SSSI) satisfying \eqref{eq:moment_estimate_SSSI}, the proof above shows that \eqref{eq:upper_bound_SSSI} holds in that case as well. By \eqref{eq:upper_bound_SSSI}, if we already know a lower bound on $\E{J_T^k}$, we get a lower bound on the Laplace transform of $X_1^*$, whereas a upper bound on the Laplace transform yields an upper bound on $\E{J_T^k}$. Since the behaviour of the Laplace transform $\E{\exp(-\lambda X_1^*)}$ as $\lambda \to \infty$ is related to that of the probability $\pr{X_1^* \le \lambda}$ as $\lambda \downarrow 0$ via Tauberian theorems, this approach could be useful to study persistence of other SSSI-processes. 
\end{remark}

\section{Some related quantities}\label{sec:fbm_rel_quant}
Given a FBM $X$, the following quantities are studied in \cite{molchan:1999a}:
\begin{align}
   \tau_{\max} & := \mathrm{argmax}_{t \in [0,1]} X_t, \label{eq:time_of_max}\\
  z_{-} & := \sup{ \set{ t \in (0,1): X_t =0 } }, \label{eq:last_zero}\\
s_+ & := \lambda(\set{t \in (0,1) : X_t > 0}), \label{eq:leb_meas_of_time_greater_zero}
\end{align}
where $\lambda$ denotes the Lebesgue measure. We remark that the definition of $\tau_{\max}$ is unambiguous since a FBM attains its maximum at a unique point on $[0,1]$ a.s.\ (\cite[Lemma~2.6]{kim-pollard:1990}). If $\xi$ denotes any of the three r.v.\ above, by \cite[Theorem~2]{molchan:1999a}, there is $c > 0$ such that
\[
  \epsilon^{1-H} \exp(-(1/c) \sqrt{  \abs{\log \epsilon} }) \precsim \pr{\xi < \epsilon} \precsim \epsilon^{1-H} \exp(c\sqrt{\abs{\log \epsilon}}), \qquad \epsilon \downarrow 0.
\]
Upon combining our results (Proposition~\ref{prop:FBM_log_boundary}), the arguments used in \cite{molchan:1999a} and the more precise estimate for the persistence probability of FBM in \cite{aurzada:2011}, we obtain the following improvement:
\begin{prop}\label{prop:FBM_related_quantities}
   If $\xi$ denotes any of the random variables in \eqref{eq:time_of_max}, \eqref{eq:last_zero} or \eqref{eq:leb_meas_of_time_greater_zero}, there is $c > 0$ such that
\begin{equation}\label{eq:lower_tail_rel_quant}
  \epsilon^{1-H} \abs{\log \epsilon}^{-c} \precsim \pr{\xi < \epsilon} \precsim \epsilon^{1-H} \abs{\log \epsilon}^c, \qquad \epsilon \downarrow 0.
\end{equation}
\end{prop}
\begin{proof}
Let us recall the relations of the probabilities involving the quantities $\tau_{\max},s_+$ and $z_-$ that are used in the proof of Theorem~2 of \cite{molchan:1999a}:\\
The symmetry and continuity of $X$ imply that
\begin{equation}\label{eq:FBM_related_quant_auxil_1}
  \pr{X_t < 0, \epsilon < t < 1} = \tfrac{1}{2} \, \pr{X_t \neq 0, \epsilon < t < 1} = \tfrac{1}{2} \, \pr{z_- < \epsilon}, \quad 0 < \epsilon <1.
\end{equation}
   Moreover, we clearly have the following inequalities:
\begin{equation}\label{eq:FBM_related_quant_auxil_2}
   \pr{X_t < 0, \epsilon < t < 1} \le \pr{s_+ < \epsilon}, \quad \pr{X_t < 0, \epsilon < t < 1} \le \pr{\tau_{\max} < \epsilon}.
\end{equation}
We will show that 
\begin{equation}\label{eq:FBM_related_quant_auxil_3}
  \epsilon^{1-H} \abs{\log \epsilon}^{-c} \precsim \pr{X_t < 0, \epsilon < t < 1} \precsim \epsilon^{1-H} \abs{\log \epsilon}^{c}.
\end{equation}
If \eqref{eq:FBM_related_quant_auxil_3} holds, \eqref{eq:FBM_related_quant_auxil_1} proves the statement for $z_-$ whereas the lower bounds in \eqref{eq:lower_tail_rel_quant} for $\xi = s_+$ and $\xi = \tau_{\max}$ follow from \eqref{eq:FBM_related_quant_auxil_2}. \\
Before establishing the remaining lower bound, let us prove \eqref{eq:FBM_related_quant_auxil_3}. Note that the self-similarity of $X$ implies that $\pr{X_t < 0, \epsilon < t < 1} = \pr{X_t < 0, 1 < t < 1/\epsilon}$. By Slepian's inequality, it holds that
\begin{align*}
   \pr{X_t < 0, 1 < t < 1/\epsilon} &\le \pr{X_t < 1, 1 < t < 1/\epsilon} \\
&\le \pr{X_t < 1, 0 < t < 1/\epsilon}/ \pr{X_t < 1, 0 < t < 1}.
\end{align*}
In view of \eqref{eq:survival_prob_FBM}, this proves the upper bound of \eqref{eq:FBM_related_quant_auxil_3}. The lower bound follows from part 2 of Proposition~\ref{prop:FBM_log_boundary} since
\[
   \pr{X_t < 0, 1 < t < 1/\epsilon} \ge \pr{X_t \le 1 - \log(1+ 3 t), 0 \le t \le 1/\epsilon}. 
\]
Let us now turn to the upper bound for $\pr{\tau_{\max} < \epsilon}$. Note that
\[
   \pr{\tau_{\max} < \epsilon} \le \pr{X_1^* < h} + \pr{X_\epsilon^* > h}, \qquad h > 0.
\]
Take $h = \epsilon^H \abs{\log \epsilon}^{\alpha}$ where $\alpha > 1/2$. Using \eqref{eq:small_value_max_FBM}, we obtain that
\[
   \pr{X_1^* < h} = \pr{X_1^* < \epsilon^H \abs{\log \epsilon}^{\alpha} } \precsim \epsilon^{1-H} \abs{\log \epsilon}^{\alpha(1 - H)/H + c + o(1)},
\]
whereas, for some constants $A,B > 0$, an application of the Gaussian concentration inequality (or Fernique's estimate stated in \cite{molchan:1999a}) yields that
\[
   \pr{X^*_\epsilon > h} = \pr{X_1^* > \epsilon^{-H} h} = \pr{X_1^* > \abs{\log \epsilon}^{\alpha}} \le A e^{-B \abs{\log \epsilon}^{2\alpha}},
\]
i.e. this term decays faster than any polynomial since $2\alpha > 1$. \\
Finally, to establish the upper bound on $\pr{s_+ < \epsilon}$, it suffices to note that the arguments in the proof of Theorem~2 of \cite{molchan:1999a} show that there is a constant $c$ such that $\pr{s_+ < \epsilon} \le 2 \pr{X_{1/\epsilon}^* < c \abs{\log \epsilon}^{1/2}}$ for all $\epsilon > 0$ small enough. It is now straightforward to conclude in view of the self-similarity and \eqref{eq:small_value_max_FBM}.
\end{proof}
\begin{remark}
   As already remarked in \cite{molchan:1999a}, $1/z_- \stackrel{d}{=} z_+ := \inf \set{s > 1: X_s = 0}$ since $(X_t)_{t > 0}$ and $(t^{2H} X_{1/t})_{t > 0}$ have the same law. Hence, Proposition~\ref{prop:FBM_related_quantities} shows that $\pr{z_+ > T}$ decays like $T^{-(1-H)}$ modulo logarithmic terms as $T \to \infty$.
\end{remark}

\bibliographystyle{amsalpha}
\bibliography{bib_FBM}
\end{document}